\documentclass[11pt,reqno]{amsart}
\usepackage{amsmath,amssymb,amsthm,amsfonts,enumerate,hyperref}
\usepackage{graphicx}
\theoremstyle{plain}
\newtheorem{theorem}{Theorem}[section]
\newtheorem{proposition}[theorem]{Proposition}

\newtheorem{lemma}[theorem]{Lemma}

\theoremstyle{definition}
\newtheorem{definition}[theorem]{Definition}

\theoremstyle{remark}

\numberwithin{equation}{section}

\newcommand{\CC}{\mathbb{C}}

\newcommand{\Riem}{\mathrm{Rm}}

\newcommand{\Ric}{\mathrm{Ric}}

\newcommand{\Hess}{\mathrm{Hess}}

\renewcommand{\div}{\mathrm{div}}

\begin{document}
\title{Variation of complex structures and the stability of K\"ahler-Ricci Solitons}
\begin{abstract}
We investigate the linear stability of K\"ahler-Ricci solitons for perturbations  induced by varying the complex structure within a fixed K\"ahler class. We calculate stability for the known examples of K\"ahler-Ricci solitons.
\end{abstract}
\author{Stuart J. Hall}

\address{Department of Applied Computing, University of Buckingham, Hunter St., Buckingham, MK18 1G, U.K.} 
\email{stuart.hall@buckingham.ac.uk}

\author{Thomas Murphy}

\address{D\'epartment de Math\'ematique,
Universit\'e Libre de Bruxelles,
CP 218,
Boulevard du Triomphe,
B-1050 Bruxelles,
Belgique.}

\email{tmurphy@ulb.ac.be}

\maketitle

\section{Introduction}

We consider a stability problem for shrinking K\"ahler-Ricci solitons. These are critical points of the $\nu$-functional, defined by Perelman  on the space of Riemannian metrics on a closed manifold $M$. The main result  is a formula for the second variation of this functional when restricted to  perturbations obtained by varying the complex structure within a fixed K\"ahler class. Such perturbations were first studied by Tian and Zhu \cite{TZ} for K\"ahler-Einstein manifolds, and our paper attempts to extend their results to K\"ahler-Ricci solitons.  Definitions and notation from the main theorem are explained below. 

\begin{theorem}[Main Theorem] \label{Maint}
Let $(M,g,f)$ be a normalised K\"ahler-Ricci soliton and let $h$ be a $f$-essential variation. The second variation of the $\nu$-functional at $g$, $\langle Nh,h\rangle_{f}$ is given as: 
$$ \langle N h,h \rangle_{f} = 2\int_{M} f \|h\|^{2}e^{-f}dV_{g}.$$ 
\end{theorem}


The main utility of this result is that if one had explicit knowledge of the metric and the function $f$ then it is possible to calculate the quantity $\langle Nh,h\rangle_{f}$ quite easily. In section 4, we do this for all the known examples of K\"ahler-Ricci solitons. Notice also that for K\"ahler-Einstein metrics $f=0$  and so $N(h)=0$, recovering a result of Tian and Zhu. \\
\\
The structure of this paper is as follows. In section 1, we begin with background on Ricci solitons and the stability problem. In section 2, the space $\mathcal{W}(g)$ and the space of $f$-essential variations studied in the above theorem are studied. We obtain  several useful characterizations of elements of these spaces. In section 3 we give a proof of the main theorem. In section 4, the stability of the known examples of Ricci solitons is investigated.\\
\\
After a preliminary version of this work was posted on the arxiv, Yuanqi Wang kindly made us aware that he had independently obtained our main Theorem \ref{Maint} as part of his Ph.D. thesis \cite{YWang} completed in 2011. The proof in \cite{YWang} is similar to ours but proceeds by direct calculation rather than using the Dai-Wang-Wei results.  His thesis also contains interesting results about convergence of the K\"ahler-Ricci flow to a K\"ahler-Einstein metric when the complex structure is allowed to vary.\\
\\
\emph{Acknowledgements:}
It is a pleasure to thank Joel Fine for useful conversations. TM is supported by an A.R.C. grant.
We acknowledge the support of a Dennison research grant from the University of Buckingham which funded research visit by TM.

\section{Ricci solitons and stability}
\subsection{Background on solitons}
Throughout this paper, $(M,g)$ is a smooth closed Riemannian manifold. 

\begin{definition}[Ricci soliton]
Let  $X\in\Gamma(TM)$ be a smooth vector field. The triple $(M,g,X)$ is called a Ricci soliton if it satisfies the equation
\begin{equation}\label{rs}
\Ric(g)+L_{X}g=cg
\end{equation}
for a constant $c \in \mathbb{R}$.
If $c<0, c=0, c>0$ then the soliton is refered to as expanding, steady and shrinking respectively. When $c\neq 0$, set $c= \frac{1}{2\tau}$. If $X=\nabla f$ for a smooth function $f$ then the soliton is called a gradient Ricci soliton and \ref{rs} becomes
\begin{equation}\label{grs}
\Ric(g)+\Hess (f)=\frac{1}{2\tau}g.
\end{equation} 
\end{definition}

When the vector field $X$ is Killing an Einstein metric is recovered; Einstein metrics are therefore referred to as trivial Ricci solitons. We can set $c=1$ to factor out homothety and as one may change the soliton potential $f$ by a constant let us  also require that
$$ \int_{M}fe^{-f}dV_{g}=0.$$
A soliton with these choices will be referred to as a \emph{normalised gradient Ricci soliton}.

As well as being interesting as generalisations of Einstein metrics, Ricci solitons also occur as the fixed points of the Ricci flow
\begin{equation}\label{RF} 
\frac{\partial g}{\partial t}=-2\Ric(g) 
\end{equation}
up to diffeomorphism. In this paper we will be considering non-trivial Ricci solitons on compact manifolds. Foundational results due to Perelman \cite{Per}  and Hamilton \cite{Ham} imply that expanding and steady Ricci solitons on compact manifolds must be trivial.  Hence our focus is on shrinking Ricci solitons. Perelman also showed that such solitons are necessarily gradient Ricci solitons. We will henceforth refer to these metrics as non-trivial shrinkers.\\
\\
Due to the work of many people \cite{Koi2}, \cite{Cao}, \cite{WangZhu}, \cite{DanWan} \cite{ps},  there are now many (infinitely many) examples of non-trivial shrinkers. One striking feature all known non-product examples share is that they are K\"ahler. This immediately implies that the vector field $\nabla f$ is holomorphic and that the underlying manifold $M$ is in fact a smooth Fano variety.\\  
\\ 
Perelman \cite{Per} showed that gradient Ricci solitons are the critical points of a functional which is usually denoted by $\nu(g)$. Let $f\in C^{\infty}(M)$ and $\tau \in \mathbb{R}$.  We say that $(f,\tau)$ is compatible if 
$$\int_{M}e^{-f}(4\pi\tau)^{-n/2}=1.$$
\begin{definition}
The $\nu$-functional  is given by
$$ \nu(g) = \inf_{\textrm{compatible }(f,\tau)}\int_{M}[(R+|\nabla f|^{2})\tau+f-n]e^{-f}(4\pi \tau)^{-n/2}dV_{g},$$
where $R$ is the scalar curvature of $g$.
\end{definition}
As well as giving a variational characterization of Ricci solitons, Perelman showed that the functional is monotonically increasing under the Ricci flow. Hence if one could perturb a soliton in a direction that increases $\nu$ and then continue the flow, one would not flow back to the soliton and the soliton would be regarded as unstable.
\subsection{Linear stability}
In order to determine the behaviour of the flow around a soliton one can investigate the second variation of $\nu(g)$ for an admissable perturbation.
\begin{definition}
Let $h\in s^2(T^*M)$. Then $g+th$, $t\in\mathbb{R}^+$ is said to be an admissable perturbation. We have $\frac{\partial g}{\partial t}\bigg|_{t=0}=h$. 
\end{definition}

If the second variation is strictly negative then  the fixed point is stable and attracting. If the second variation has positive directions then one may perturb the soliton and then flow away. Natasha Sesum has obtained fundamental results on this topic \cite{Ses}. 

\begin{proposition}[Cao-Hamilton-Ilmanen, Cao-Zhu, \cite{CHI}, \cite{CMZ}]
Let $h\in s^{2}(TM^{\ast})$ be an admissable variation of a Ricci soliton $g$.  The second variation of $\nu$ is given by
$$D^{2}_{g}\nu(h,h) = \int_{M}\langle Nh,h\rangle e^{-f}dV_{g},$$
 where
\begin{equation}\label{Neqn}
Nh = \frac{1}{2}\Delta_{f} h+\Riem (h,\cdot)+\div^{\ast}\div_{f}h+\frac{1}{2}\Hess(v_{h})+C(h,g)\Ric.
\end{equation}
Here $\Delta_{f}(\cdot) = \Delta(\cdot) - \nabla_{\nabla f}(\cdot)$, $\div_{f}(\cdot) = \div(\cdot) - \iota_{\nabla f}$, $v_{h}$ is the solution of the equation
$$ \Delta_{f}v_{h}+\frac{v_{h}}{2\tau} = \div_{f}\div_{f}(h), $$ 
and 
$$C(h,g) = \frac{\int_{M}\langle \Ric, h \rangle e^{-f}dV_{g}}{\int_{M}Re^{-f}dV_{g}}.$$
\end{proposition}
This operator allows us to define the concept of linear stability.
\begin{definition}
Let $(M,g,f)$ be a Ricci soliton. The soliton is \emph{linearly stable} if the operator $N$ is non-positive definite and \emph{unstable} otherwise.
\end{definition}
We now focus upon K\"ahler-Ricci solitons. The first observation regarding stability is the following:
\begin{theorem}[\cite{CHI,HM, TZ}]\label{KRSinst}
Let $(M,g,f)$ be a K\"ahler-Ricci soliton. If \newline $dim H^{(1,1)}(M)>1$, then $(M,g,f)$ is unstable. 
\end{theorem}
K\"ahler-Ricci solitons can be viewed as fixed points of a flow related to the Ricci flow (\ref{RF}) called the K\"ahler-Ricci flow, which in the Fano case can be written as
\begin{equation}\label{KRF}
\frac{\partial g}{\partial t} = -\Ric (g)+g,  \ \  g(0)=g_{0}.
\end{equation} 
One important point about this flow is that it preserves the K\"ahler class. A foundational result about this flow, due to Cao \cite{Caoinv}, is that it exists for all time. The convergence of it is an extremely subtle issue because the complex structure can jump in the limit at infinity. Hence the type of convergence one expects is rather weak. This is illustrated by the following example. 
\begin{theorem}[Tian-Zhu, \cite{TZJams}]
Let $M$ be a compact manifold which admits a K\"ahler-Ricci soliton $(g_{KRS},f)$. Then any solution of (\ref{KRF}) will converge to $g_{KRS}$ in the sense of Cheeger-Gromov if the initial metric $g_{0}$ is invariant under the maximal compact subset of the automorphism group of $M$.
\end{theorem}
The unstable perturbations in Theorem \ref{KRSinst} do not preserve the canonical class. Therefore, from the point of view of the K\"ahler-Ricci flow it is natural to consider perturbations which fix the K\"ahler class but allow the complex structure of the manifold to vary. This  was initiated by Tian and Zhu \cite{TZ}. 
\begin{definition}
Let $(M,g_{KRS})$ be a K\"ahler-Ricci soliton with complex structure $J_{KRS}$.  The space of perturbations $\mathcal{W}(g_{KRS})$ is defined as follows:
\begin{align*}
\mathcal{W}(g_{KRS}) = \bigg\lbrace h & \in s^{2}(TM^{\ast}) \ | \ \textrm{there is a family of K\"ahler metrics } (g_{t},J_{t})\\ 
 &\textrm{with } \frac{\partial g_{t}}{\partial t}\bigg|_{t=0} = h, [g_{t}(J_{t}\cdot,\cdot)]=c_{1}(M,J_{KRS}),\\ &   (g_{0},J_{0})=(g_{KRS}, J_{KRS})\bigg\rbrace.
\end{align*} 
\end{definition}  

The following result was our main motivation for considering this space of perturbations:
\begin{theorem}[Tian-Zhu, \cite{TZ}] \label{TZt}
Let $(M,g_{KE})$ be a K\"ahler-Einstein metric and let $h \in \mathcal{W}(g_{KE})$. Then
$$\langle N(h),h \rangle_{f} \leq 0.$$ 
\end{theorem}
Tian and Zhu then conjectured that a similar result should be true for Ricci solitons. Our formula in the Theorem \ref{Maint} shows that this might not be true in general.  The integral in the main theorem does not seem to have a sign in general. However, the examples we calculate in section 4 do all have $\langle N(h),h\rangle_{f}=0$; this seems be an artifact of their construction rather than a manifestation of some result in complex differential geometry.
\\
We mention here the related study of stability by Dai, Wang and Wei \cite{DWW}.  They prove that Kahler-Einstein metrics with negative scalar curvature are stable. There is also the recent work of Nefton Pali \cite{Pali} in this area. He considers a related functional known in the literature as the $W$-functional (here one is free to pick a volume form whereas in the definition  of the $\nu$-functional one is determined by the metric).

\subsection{Notation and convention}
We use the curvature convention that $\Riem(X,Y)Z = \nabla_{Y}\nabla_{X}Z-\nabla_{X}\nabla_{Y}Z+\nabla_{[X,Y]}Z$. The convention for divergence that we adopt is $\div(h) = tr_{12}(\nabla h)$, the rough Laplacian $$\Delta h = \div(\nabla h) = -\nabla^{\ast}\nabla h$$ is then negative definite.
Set
$$
\langle \cdot, \cdot\rangle_f = \int_M \langle \cdot, \cdot\rangle e^{-f}dV_g
$$
to be the twisted inner product on tensors at a Ricci soliton $(M,g,f)$. We will denote pointwise inner products induced on tensor bundles by $g$ with round brackets $(\cdot, \cdot)$. The adjoint of a differential operator (such as $\nabla$) with respect to this inner product will be denoted with a subscipt $f$ (i.e. $div_f$) throughout. 

\section{Background on variations of complex stucture}
\subsection{Variations of complex structure}
We recall that an almost complex stucture on a manifold $M$ is a section $J$ of the endomorphism bundle $End(TM)$ satisfying $J^{2}=-id$.  For $M$ to be a complex manifold we require that the complex structure is \emph{integrable}. By the Newlander-Nirenberg theorem we may take integrable to mean that the Nijenhuis tensor $\mathcal{N}(J)=0$. We will be concerned with infinitesimal variations of complex structure that are modelled on those coming from a one parameter family of complex structures $J_{t}$.  As we are only working at an infinitesimal level, we don't actually mind if our variations are induced by such a family. 

\begin{definition}[Infinitesimal variation of complex structure]
Let $(M,g,J)$ be a K\"ahler manifold, a tensor $\zeta \in End(TM)$ is called an infinitesimal variation of complex structure if it satisfies the two equations:
\begin{equation}\label{vComp}
\zeta J+J\zeta=0,
\end{equation} 
\begin{equation}\label{vint}
\dot{\mathcal{N}}(\eta)=0.
\end{equation}
\end{definition}
The equation (\ref{vComp})  simply says that the $J_{t}$ are almost complex structures, the equation (\ref{vint}) comes from requiring that they are integrable. In the above definition we are viewing $\zeta$ as a section of the bundle $End(TM)$ which is defined for any manifold.  Switching in the usual manner to the complex viewpoint, Equation (\ref{vComp}) can be thought of as saying that $\zeta$ is a section of the bundle $\Lambda^{(0,1)}\otimes TM^{(1,0)}$. We will variously view the variation as an element of the real bundle $End(TM)$, a section of the bundle $\Lambda^{(0,1)}\otimes TM^{(1,0)}$ and, using the metric to lower indices, as a section of $TM^{\ast} \otimes TM^{\ast}$ and $\Lambda^{(0,1)} \otimes \Lambda^{(0,1)}$.  We note that in complex coordinates equations (\ref{vComp})  and (\ref{vint}) become
 $$\zeta_{\alpha}^{\beta}=0 \text{  and } \nabla_{\alpha}\zeta_{\beta \gamma} = \nabla_{\beta}\zeta_{\alpha \gamma}.$$
The bundle $\Lambda^{(0,1)} \otimes TM^{(0,1)}$ is an element of the Dolbeault complex
$$TM^{(1,0)} \stackrel{\bar{\partial}}{\rightarrow} \Lambda^{(0,1)}\otimes TM^{(1,0)} \stackrel{\bar{\partial}}{\rightarrow} \Lambda^{(0,2)}\otimes TM^{(1,0)} \stackrel{\bar{\partial}}{\rightarrow} ...,$$
where $\bar{\partial}$ is the usual d-bar operator associated to a holomorphic vector bundle over a complex manifold. Equation (\ref{vint}) is equivalent to requiring that $\bar{\partial}\zeta =0$.

Analogous to Tian-Zhu \cite{TZ} and following Koiso \cite{Koi1}, we will decompose the space of infinitesimal variations into \emph{trivial} variations and \emph{$f$-essential} variations.

By analogy with the twisted inner product, set 
$$
 \Delta_{\overline{\partial},f} := \overline{\partial}\overline{\partial}_f^* + \overline{\partial}_f^*\overline{\partial} $$
to be the twisted $\partial$-Laplacian. 

\begin{definition}[$f$-essential variation]
Let $\zeta$ be an infinitesimal variation of the complex structure $J$.  We say $\zeta$ is trivial if $\zeta = L_{Z}J$ for a smooth vector field $Z \in TM$.  A variation $\zeta$ is said to be $f$-essential if 
$$\int_{M}\langle\zeta, L_{Z}J\rangle e^{-f}dV_{g}=0$$
for all $Z\in \Gamma(TM)$.
\end{definition} 
The following lemma gives a useful characterisation of $f$-essential varitations.
\begin{lemma}[cf. Lemma 6.4 in \cite{Koi1}]
Let $\zeta$ be a $f$-essential variation and let $h(\cdot,\cdot) = \omega(\cdot,\zeta\cdot) $.  If $h$ is symmetric, then
\begin{enumerate}
\item $\bar{\partial}_{f}^{\ast}\zeta =0$, 
\item  $\div_{f}h = 0.$
\end{enumerate}
In particular, a $f$-essential variation is $\Delta_{\bar{\partial},f}$-harmonic. 
\end{lemma}
\begin{proof}

1) As $\zeta$ is $f$-essential, 
$$
\int_M \langle L_ZJ, \zeta \rangle e^{-f} =0
$$
for all $Z \in \Gamma(TM)$. The Lie derivative of the complex stucture is related to the $\bar{\partial}$-operator by
$$\overline{\partial}_{\cdot} Z = -\frac{1}{2}JL_ZJ(\cdot).$$
Hence, up to a constant, $\langle L_ZJ, \zeta\rangle_f = \langle \overline{\partial} Z, \zeta \rangle_f$ and $\overline{\partial}^*_{f} \zeta = 0$, as claimed.\\
\\
2) We begin by noting that $\zeta$ being $f$-essential means that 
$$\langle L_{Z}J,\zeta\rangle_{f}= \langle \omega(\cdot, L_{Z}J(\cdot)),h\rangle_{f}=0.$$
Rewriting and using the Cartan formula we have
$$\omega(\cdot, L_{Z}J(\cdot)) = L_{Z}g(\cdot,\cdot)-L_{Z}\omega(\cdot,\cdot) =2div^{\ast}Z^{\flat}(\cdot,\cdot))-(d\circ\iota_{Z}\omega)(\cdot,J\cdot).$$
The result follows by noting that
$$\langle(d\circ\iota_{Z}\omega)(\cdot,J\cdot),h\rangle_{f} = -\langle(d\circ\iota_{Z}\omega)(\cdot,\cdot),h(\cdot,J\cdot)\rangle_{f},$$
and that $h(\cdot,J\cdot)$ is symmetric. 
\end{proof}
In the previous lemma we have assumed that $h$ is symmetric.  This is not strictly necessary on Fano manifolds as one can show that an antisymmetric, $J$-anti-invariant 2-tensor defines a global holomorphic 2-from.  Then one can appeal to a classical result of Bochner to show that on a Fano manifold such a form is zero (c.f. \cite{Besse} 11.24). Tian and Zhu give a straightforward proof of this fact in the case one is at a K\"ahler-Einstein metric \cite{TZ}.

Tian and Zhu \cite{TZ} decompose the space $\mathcal{W}(g)$ modulo the
action of the diffeomorphism group.  They show that
$$ \mathcal{W}(g)/\mathcal{D}(M) = \mathcal{A}^{(1,1)}\bigoplus H^{1}(M,TM)$$
where $\mathcal{A}^{(1,1)}$ is the space of $\partial\bar{\partial}$-exact
(1,1)-forms and $H^{1}(M,TM)$ is the usual cohomology for the holomorphic
vector bundle $TM$.  Tian and Zhu then show that for a general K\"abler-Ricci soliton, $N |_{\mathcal{A}^{(1,1)}}\leq 0$ so that potentially destabilising elements of
$\mathcal{W}$ actually lie in $H^{1}(M,TM)$ (they then show that $N$
vanishes on this space when $g$ is an Einstein metric). Hence we will only
consider perturbations in $H^{1}(M,TM)$ and we will use the special
representatives given by $f$-essential perturbations. Formally we have:
\begin{proposition}[Tian-Zhu, \cite{TZ}]
Let $(M,g_{KRS},J)$ be a K\"ahler-Ricci soliton. Then we have the following decomposition
$$\mathcal{W}(g_{KRS})/\mathcal{D}(M) \cong \mathcal{A}^{(1,1)}(M,J) \bigoplus H^{1}(M,TM).$$
where $\mathcal{D}(M)$ is the diffeomorphism group of $M$.
The operator $N$ is non-positive when restricted to $\mathcal{A}^{(1,1)}(M,J).$
\end{proposition}

\section{Proof of main Theorem}
Consider a $f$-essential variation of the complex structure $h \in H^{1}(M,TM)$. Firstly, as $h$ is $J$-anti-invariant it is apparent that $C(h,g)=0$. Thus
$$ \langle N(h),h\rangle_{f} = \langle \frac{1}{2}\Delta_{f} h+\Riem(h,\cdot),h\rangle_{f}.$$

In order to evaluate the above we will use a Weitzenb\"ock formula. In order to explain the formula we will digress briefly into the spinorial construction used in \cite{DWW}. This is a powerful generalisation of the techniques used by Koiso in \cite{Koi1}. 

As $M$ is Fano it has a canonical $\textrm{spin}^{c}$ structure and parallel spinor $\sigma_{0}\in \Gamma(\mathcal{S}^{c})$ where $\mathcal{S}^{c}\rightarrow M$ is the $\textrm{spin}^{c}$ spinor bundle. This induces a map 
\begin{align*}
&\Phi : s^{2}(TM^{\ast})\rightarrow \mathcal{S}^{c}\otimes TM^{\ast},\\
&\Phi(h) = h_{ij}e_{i}\cdot\sigma_{0}\otimes e^{j},
\end{align*} 
where $\{e_{i}\}$ is a orthonormal basis of $TM$ and $e_{i}\cdot\sigma_{0}$ denotes Clifford multiplication in $\mathcal{S}^{c}$. 

For $1\leq i\leq m$, following \cite{DWW} choose
$$
X_i = \frac{e_i - \sqrt{-1}Je_i}{\sqrt{2}} ,  \bar{X_i}= \frac{e_i + \sqrt{-1}Je_i}{\sqrt{2}}.
$$
Then $\lbrace X_1, \dots, X_m\rbrace$ is a local unitary frame for $T^{1,0}M$. Set $\lbrace \theta^1, \dots, \theta^m\rbrace$ to be its dual frame. Then 
$$
\Phi(h) = h(\bar{X_i},\bar{X_j}) \bar{\theta^i}\otimes \bar{\theta^j}. 
$$
This can be identified with 
$$
\Psi(h) = h(\bar{X_i},\bar{X_j})\bar{\theta^i}\otimes X_j \in \wedge^{0,1}(TM)
$$
where $TM$ is the holomorphic tangent bundle. 

\begin{lemma}[Dai-Wang-Wei \cite{DWW1} Lemma No 2.3]\label{l1}
For $h,\tilde{h} \in s^{2}(TM^{\ast})$,
$$Re(\Phi(h),\Phi(\tilde{h})) = ( h,\tilde{h}).$$
\end{lemma}

Recall that $\mathcal{S}^c$ has an induced K\"ahler structure $J_c$. Choosing normal coordinates at a point, it is not hard to see that 
$(\Delta_{\overline{\partial}_c}\circ \Psi)(h) = (\Psi\circ\Delta_{\overline{\partial}})(h)$.
Moreover, under this  identification of $\Phi(h)$ with $\Psi(h)$, the Dirac operator $D$ is identified with $\sqrt{2}(\overline{\partial} - \overline{\partial}^*)$. Thus $D^*D$ is identified with $-2 \Delta_{\overline{\partial}}$.

The main result we need is the following Weitzenb\"ock formula:
\begin{lemma}[Dai-Wang-Wei \cite{DWW} Lemma 2.3]\label{l2}
Let $h \in s^{2}(TM^{\ast})$ and let $D$ be the Dirac operator. Then
\begin{equation}\label{DWeit}
D^{\ast}D(\Phi(h)) = \Phi(\nabla^{\ast}\nabla h - 2\Riem (h,\cdot) +\Ric \circ h - h\circ i\rho)
\end{equation}
where $\rho$ is the Ricci form. 
\end{lemma}
In order to deal with the Ricci curvature terms we use the following lemma which is implicit in the proof of Theorem 2.5 in \cite{DWW}.
\begin{lemma}
Let $h$ be a skew-hermitian section of $s^{2}(TM^{\ast})$. Then
$$(Ric\circ h-h\circ i\rho, h) =0.$$ 
\end{lemma}
\begin{proof}
This is a pointwise calculation. Choose normal coordinates at $p\in M$,  $\{e_{1},...,e_{2m}\}$ where $e_{m+i}=Je_{i}$ for $1\leq i \leq m$. We can also choose this basis so that the Ricci tensor is diagonalised i.e. $Ric(e_{i},e_{j})=c_{i}\delta_{ij} $ where $c_{m+i}=c_{i}$.
We have
$$
Re(\Phi(\Ric \circ h),\Phi(h) ) = \sum_{i,j=1}^{2m}c_{i}h_{ij}^{2}, 
$$
\begin{align*}
-Re(\Phi(h \circ F),\Phi(h) ) = -2\sum_{j=1}^{m}\sum_{i=1}^{m}c_{j}& (h_{(i+m)j}h_{i(j+m)}\\
&-h_{ij}h_{(i+m)(j+m)}).
\end{align*}
If $h$ is skew-Hermitian then 
$$h_{ij}=-h_{(i+m)(j+m)} \text{ and } h_{i(j+m)}=h_{(i+m)j}.$$
Hence
$$-Re(\Phi(h \circ F),\Phi(h)) = -2\sum_{i=1}^{m}\sum_{j=1}^{2m}c_{j}(h_{ij}^{2})=-\sum_{i,j=1}^{2m}c_{i}h_{ij}^{2}, $$
and the result follows.
\end{proof}

The final lemma we need to prove the main result in this section is a technical lemma to deal with the extra term one obtains by using the rescaled volume form $e^{-f}dV_{g}$.

\begin{lemma}\label{l3}
Let $A \in \Omega^{1}(M)$ be a one-form and $B\in \bigotimes^{k}TM^{\ast}$
\begin{equation}\label{divprod1}
\div (A\otimes B) = \div (A)\otimes B+\nabla_{A^{\sharp}}B
\end{equation}
\begin{equation}\label{divprod2}
div(df\otimes h) = (\Delta f) h + \nabla_{\nabla f}h 
\end{equation}
\begin{equation}\label{Nhh}
-\langle \nabla_{\nabla f} h,h\rangle_{f} = \frac{1}{2}\int_{M} \Delta_{f}f\|h\|^{2}e^{-f}dV_{g}. 
\end{equation}
\end{lemma}

\begin{proof}
For (\ref{divprod1}) we calculate using a normal, orthonormal basis $\{e_{i}\},$
$$div(A\otimes B) = \nabla_{e_{i}}(A \otimes B)(e_{i},\cdot) = div(A)\otimes B+\nabla_{A^{\sharp}}B.$$
For (\ref{divprod2}) we use $A=df, B=h$ in (\ref{divprod1}). In order to prove (\ref{Nhh}) we note that
$$\langle\nabla_{\nabla f}h,h\rangle_{f} = \langle \iota_{\nabla f} \nabla h,h\rangle_{f} = \langle \nabla h, df\otimes h\rangle_{f}=-\langle h,\div_{f}(df\otimes h)\rangle_{f}. $$
Now using (\ref{divprod2}) we have
$$\langle \nabla_{\nabla f} h,h\rangle_{f} = \int_{M}|\nabla f|^{2}\|h\|^{2}e^{-f}dV_{g}-\langle h,\div (df \otimes h)\rangle_{f}.$$
$$ = -\int_{M}(\Delta_{f} f)\|h\|^{2}e^{-f}dV_{g}-\langle \nabla_{\nabla f} h,h\rangle_{f}$$
and the result follows.

\end{proof}

As noted in \cite{HM}, the soliton potential function  of a normalised gradient Ricci soliton solves the equation 
$$\Delta_{f}f = -2f.$$

\begin{proof}[Proof of Main Theorem:] Lemmas \ref{l1} and \ref{l2} yield that pointwise
\begin{align*}
(\frac{1}{2}\Delta h+\Riem (h, \cdot), h) &= Re( \Phi(\frac{1}{2}\Delta h+\Riem (h, \cdot)), \Phi(h))\\
&= Re( D^*D(\Phi(h), \Phi(h))) \\
& = Re( -2\Delta_{\overline{\partial}_c}(\Psi(h), \Psi(h))\\
&= -2 Re ( \Phi(\Delta_{\overline{\partial}}h), \Phi(h))\\
&= -2( \Delta_{\bar{\partial}}h,h ). 
\end{align*}
However as $h$ is $f$-essential then it is orthogonal to the image of $\Delta_{\bar{\partial}}$ with respect to the global inner product.
Hence
$$\int_{M} ( \frac{1}{2}\Delta h +\Riem (h,\cdot),h ) e^{-f}dV_{g} = 0$$
and the result follows.  \end{proof}

\section{Examples and applications}

\subsection{Setup}
As mentioned in the introduction, there are three main sources for concrete examples of K\"ahler-Ricci solitons: the Dancer-Wang, Podesta-Spiro and the Wang-Zhu examples. The Wang-Zhu solitons exist on toric-K\"ahler manifolds and are non-trivial precisely when the Futaki invariant is non-zero.  Unfortunately, this class of manifold does not admit any non-trivial deformations of complex structure.  This follows from
\begin{theorem}[Bien-Brion,\cite{BB} Theorem 3.2]
Every Fano toric-K\"ahler manifold $M$ has $H^{1}(M,TM)=0$. 
\end{theorem}

Similarly, one can see the Podesta-Spiro examples are rigid. The next class of examples to investigate are provided by the Dancer-Wang solitons.  These solitons are generalisations of the soliton on $\CC\mathbb{P}^{2}\sharp\overline{\CC\mathbb{P}}^{2}$ constructed by Koiso \cite{Koi2} and Cao \cite{Cao}. We begin by reviewing their construction.\\
\\
Let $(V_{i},r_{i},J_{i})$, $1 \geq i \leq r$ be Fano K\"ahler-Einstein manifolds with first Chern class $c_{1}(V_{i},J_{i}) = p_{i}a_{i}$, where $p_{i}$ are positive integers and $a_{i} \in H^{2}(V_{i};\mathbb{Z})$ are indivisible classes. The K\"ahler-Einstein metrics $r_{i}$  are normalised so that $\Ric (r_{i}) = p_{i}r_{i}.$
For $q=(q_{1},...,q_{r})$ with $q_{i} \in \mathbb{Z}-\{0\}$, let $P_{q}$ the total space of the principal $U(1)$-bundle over $B:=V_{1}\times V_{2}\times ...\times V_{r}$ with Euler class $\sum_{1}^{r}q_{i}\pi_{i}^{\ast}a_{i}$ where $\pi_{i}: V_{1} \times ... \times V_{r} \rightarrow V_{i}$ is the projection onto the ith factor. Denote by $M_{0}$ the product $I \times P_{q}$ for the unit interval $I$. We denote by $\theta$ the principal $U(1)$ connection on $P_{q}$ with curvature 
$$\Omega :=\sum_{1}^{r}q_{i}\pi^{\ast}\eta_{i}$$
where $\eta_{i}$  is the Kahler form of $r_{i}$.  There is a one-parameter family of metrics on $P_{q}$ given by
$$g_{t} := f^{2}(t)\theta \otimes \theta +\sum_{i=1}^{r}l_{i}^{2}(t)\pi_{i}^{\ast}r_{i}$$
where $f$ and $l_{i}$ are smooth functions on $I$ with prescribed boundary behaviour. Finally,  consider the metric on $M_{0}$ given by
$$g = dt^{2}+g_{t},$$ 
with the correct boundary behaviour of $f$ and the $l_{i}$. This metric then extends to a metic on a compactification of $M_{0}$, which we denote $M$.

The complex structure on this manifold can be described explicitly by lifting the complex structure on the base and requiring that $J(N) = -f(t)^{-1}Z$ where $N=\partial_{t}$ is normal to the hypersurfaces and $Z$ is the Killing vector that generates the isometric $U(1)$ action on $P_{q}$.

\subsection{Deformations of Dancer-Wang solitons}

The Ricci soliton equations in this setting reduce to a system of ODEs. We have the following existence theorem:
\begin{theorem}[Dancer-Wang, \cite{DanWan} Theorem 4.30]
Let $M$ denote the compactification of $M_{0}$ as above. Then $M$ admits a K\"ahler-Ricci soliton $(M, g, u)$, which is Einstein if and only if the associated Futaki invariant vanishes.
\end{theorem}
We refer to \cite{DanWan} for details of the constructions. If one chooses the components $V_{i}$ to be homogenous, K\"ahler-Einstein manifolds then the resulting $M$ is toric. However, by choosing the components $V_{i}$ to be non-homogenous, Fano, K\"ahler-Einstein and calculating the Futaki invariant,  they give examples of non-toric K\"ahler-Ricci solitons.  It is these that may admit complex deformations.   

Suppose that $V_{i}$ is a Fano, K\"ahler-Einstein manifold admitting deformations of its complex structure $J_{i}$.  We consider   variations $h_{i,t}$ in the Kahler metric $r_{i,t}$ such that the K\"ahler form $\eta_{i,t} = r_{i,t}(J_{i,t}\cdot,\cdot)$ remains in the class $c_{1}(V_{i},J_{0})$. This induces a variation in the metric on the whole space given by
$$ h_i = l_{i}^{2}(t)\pi^{\ast}h_{i,t}.$$

Clearly the same procedure works for any product of K\"ahler-Einstein manifolds with some (or all) of the factors admitting complex deformations. Here it is simply stated for one factor for simplicity.  Let us state our final result:
\begin{theorem}
For this perturbation $h$, one has $N(h)= 0$. 
\end{theorem}

\begin{proof}
It follows from the construction of $h$ that 
$$
\|h_i\|_g^{2}=\int_I\|h_{i,t}\|^{2}_{r_{i,t}}dt = \|h_{i,0}\|^2_{r_{i,0}}\int_Idt = \|h_{i,0}\|^2_{r_{i,0}}.
$$ 
In other words $\|h_{i,t}\|_{r_{i,t}}$ is independent of $t$. We now see that
\begin{align*}
\langle Nh, h \rangle  =& \int_{M}\Delta_uu\|h\|_g^{2} e^{-u}dV_{g} \\
 =& \|h\|_g^2\int_M (\Delta_u u e^{-u})  dV_g\\
 =& 0.
 \end{align*}
\end{proof}

\textbf{Remark} The significance of this result is that it verifies Tian-Zhu's conjecture for every obvious example of a complex deformation of  the known K\"ahler-Ricci solitons. We do not know of any explicit deformations beyond these.

It is notable that all $f$-essential perturbations $h$ known to us one has\newline  $N(h) = 0$: understanding if this is always the case would involve calculating \newline $H^1(M,TM)$, which is not easy to calculate in general.

\end{document}